\documentclass[a4paper,10pt]{article}

\usepackage[english]{babel}
\usepackage{amsmath,amsfonts,amssymb,amsthm}

\usepackage[top=2cm, bottom=2cm, left=2.5cm, right=2.5cm]{geometry}

\usepackage{indentfirst}
\usepackage[T1]{fontenc}

\newtheorem{theorem}{Theorem}[section] 
\newtheorem{lemma}[theorem]{Lemma}     
\newtheorem{definition}[theorem]{Definition}
\newtheorem{corollary}[theorem]{Corollary}
\newtheorem{proposition}[theorem]{Proposition}

\numberwithin{equation}{section}

\newcommand{\beq}{\begin{equation}}
\newcommand{\eeq}{\end{equation}}
\newcommand{\bdm}{\begin{displaymath}}
\newcommand{\edm}{\end{displaymath}}

\newcommand{\C}{\mathbb C}
\newcommand{\Cd}{\overline{\mathbb C}}
\newcommand{\N}{\mathbb N}
\newcommand{\Z}{\mathbb Z}
\newcommand{\R}{\mathbb R}

\newcommand{\M}{\mathcal M}
\newcommand{\NN}{\mathcal N}
\renewcommand{\O}{\mathcal O}
\newcommand{\HS}{\mathcal{H}}
\newcommand{\W}{\mathcal W}

\newcommand{\PS}{\mathcal{P}}
\newcommand{\p}{\wp_{\Lambda}}

\newcommand{\dist}{\mathrm {dist}}
\newcommand{\diam}{\mathrm{diam}}

\newcommand{\Crit}{\mathrm {Crit}}
\newcommand{\e}{\mathrm{e}}

\begin{document}

\begin{center}
\Large
\textbf{Misiurewicz parameters for Weierstrass elliptic functions\\ based on triangle and square lattices}\\
\vspace{1em}
\large
Agnieszka Bade\'nska

\smallskip
\footnotesize
Faculty of Mathematics and Information Science\\
Warsaw University of Technology\\
ul. Koszykowa 75\\
00-662 Warszawa\\
Poland\\
badenska@mini.pw.edu.pl
\end{center}

\normalsize
\begin{abstract}
For two families of Weierstrass elliptic functions - based on triangular or square lattices - we prove that the~set of Misiurewicz parameters has the~Lebesgue measure zero in~$\C$.
\end{abstract}

\section{Introduction}

We consider Weierstrass elliptic functions based on the~lattice
\bdm \Lambda=\{m\lambda_1+n\lambda_2\colon m,n\in\Z\}=:[\lambda_1,\lambda_2]\;,\; \lambda_2/\lambda_1\notin\R, \edm
given by the formula
\bdm \p(z)=\frac{1}{z^2}+\sum_{\omega\in\Lambda\setminus\{0\}}\left(\frac{1}{(z-\omega)^2}-\frac{1}{{\omega}^2}\right). \edm
It is a~wide class of meromorphic functions, periodic with respect to $\Lambda$ and of order two. We refer to \cite{hk1,hk} for a~nice description of dynamical and measure theoretic properties of~$\p$ depending on the lattice~$\Lambda$ as well as investigation of some specific parametrized families of Weierstrass elliptic functions.
For an~introduction to the~theory of iterating complex functions see e.g.~\cite{berg}.

Even fixing type of the~lattice $\Lambda$, i.e. the shape $\tau=\lambda_2/\lambda_1$ of the corresponding period parallelogram of~$\p$, we still obtain an~incredible richness of dynamical behaviour and properties of Weierstrass functions. We are particularly interested in two families of functions: based on triangular lattices, i.e. satisfying $\e^{2\pi{i}/3}\Lambda=\Lambda$, and on square lattices, i.e. such that $i\Lambda=\Lambda$.  Let us specify the~families $\W_t$ and $\W_s$ we are interested in.

The family $\W_t$ consists of all Weierstrass elliptic functions based on triangular lattices. It can be given by:
\bdm  \W_t=\left\{f_{\lambda}:=\wp_{\Lambda_{\lambda}}\colon\C\to\Cd,\; \textrm{where} \; \Lambda_{\lambda}=[\lambda,\e^{2\pi{i}/3}\lambda],\; \lambda\in\C\setminus\{0\} \right\}. \edm
All Weierstrass elliptic functions based on square lattices are members of the family $\W_s$ defined by:
\bdm  \W_s=\left\{f_{\lambda}:=\wp_{\Lambda_{\lambda}}\colon\C\to\Cd,\; \textrm{where} \; \Lambda_{\lambda}=[\lambda,\lambda{i}],\; \lambda\in\C\setminus\{0\} \right\}. \edm
Since most of the considerations is the same for both families, we are not to restrictive about the~notation. If it is important, we will point out the differences. 

The dynamics of these functions is rather rigid because of the~close relationship between trajectories of critical values. Therefore, there are only couple of possible structures of the~Fatou set that may occur -- we will list them in the next section (Lemma~\ref{Wtcases} and Lemma~\ref{Wscases}). In this paper we will show that one of the cases, i.e. when $f_{\lambda}$ satisfies so-called Misiurewicz condition, appears very rarely. 

The notion of Misiurewicz maps derives from the paper \cite{mm} by M.~Mi\-siu\-re\-wicz, where the author studied e.g. the real quadratic family $g_a(x)=1-ax^2$ in the case when $g_a$ is non-hyperbolic and the critical point $0$ is non-recurrent. We refer to \cite{a} for a~nice discussion concerning various definitions of Mi\-siu\-re\-wicz condition in the complex case and more references. For the~considered families of Weierstrass elliptic function we introduce the following definition.
\begin{definition}\label{defMis}
A~function $f_{\lambda}$ from the~family $\W_t$ or $\W_s$ satisfies the~Misiurewicz condition (equivalently $\lambda$ is a~Misiurewicz parameter) if all singular values of $f_{\lambda}$ belong to the~Julia set and the~set $\PS(f_{\lambda})\cap\C$, i.e. the~finite part of the~postsingular set, is bounded and disjoint from the~set $\Crit(f_{\lambda})$ of the~critical points of~$f_{\lambda}$.
\end{definition}
In other words every singular value of $f_{\lambda}$ is either a~prepole or has a~bounded trajectory staying in a~positive distance from the~set of critical points~$\Crit(f_{\lambda})$.
This may seem more restrictive than the~definition introduced by Graczyk, Kotus and \'Swi\k{a}tek in~\cite{gks} as we demand that all singular values lie in the~Julia set but after analysis of dynamics of functions from the~considered families it will be clear that the~above definition is natural in this case.
Note also that the~definition includes the case (sometimes referred as pure Misiurewicz) when all singular values are preperiodic.

It was proved by M.~Aspenberg in \cite{a} that the set of Misiurewicz maps has the~Lebesgue measure zero in the space of rational functions of any fixed degree. Next, this result was extended in \cite{b} to the exponential family which is one dimensional space of entire transcendental maps. In this~paper we generalize these results and prove the~following.
\begin{theorem}\label{main}
For the~families $\W_t$ and $\W_s$ the~set of Misiurewicz parameters has the~Lebesgue measure zero in~$\C$.
\end{theorem}

We will prove this result in two steps. First we deal with parameters to which we can apply similar technique as in~\cite{a,b} and show that the~following is true.
\begin{theorem}\label{A}
For the families $\W_t$ and $\W_s$ the set of parameters $\lambda$ for which there exists in the Julia set~$J(f_{\lambda})$ a~critical value which is not a~prepole and has a~bounded trajectory not accumulating on the~critical set $\Crit(f_{\lambda})$ has the~Lebesgue measure zero in~$\C$.
\end{theorem}

Because of the~close relationship between all critical trajectories in the considered families the~assumptions of Theorem~\ref{A} imply in particular that all critical values of~$f_{\lambda}$ (except for the~pole $0$ in the~case of a~square lattice) are not prepoles and have bounded trajectories in~$J(f_{\lambda})$ separated from $\Crit(f_{\lambda})$, so in fact $f_{\lambda}$ is a~special case of a~Misiurewicz map.

However, in order to deal with all Misiurewicz parameters we need to consider one more case, i.e. when all critical values of~$f_{\lambda}$ are prepoles. Therefore, we will prove at the~end the~following lemma.
\begin{lemma}\label{critpole}
For the~families $\W_t$ and $\W_s$ the~set of parameters~$\lambda$ for which all critical values of~$f_{\lambda}$ are prepoles is countable.
\end{lemma}
Note that Theorem~\ref{A} and Lemma~\ref{critpole} imply the~main result of the~paper, i.e. Theorem~\ref{main}, since elliptic functions have no asymptotic values.

The proof of Theorem~\ref{A} in general follows the Aspenberg's approach from~\cite{a}, repeated with some changes in \cite{b} for the~exponential family. Note however, that our case brings new difficulties. We have to deal not only with infinite degree of maps and essential singularity at~$\infty$ but also with prepoles which become essential singularities in~$\C$ for iterates of considered functions. That is why we have to be sure that we can stay away from poles and essential singularities in order to proceed with calculations. Some minor but crucial changes had to be done especially in the~section~\ref{secholmot} where we prove existence of a~holomorphic motion and so-called transversality condition
and for measure estimates in a~big scale in the~section~\ref{secmeasest} (see Lemma~\ref{measest}).

Lemma~\ref{critpole} is proved at the~end of the~paper. We describe the~condition that all critical values are prepoles by an~analytic equation depending on a~countable number of parameters (this is possible because of the~close relation between critical values of considered functions). Next, using postsingular stability, $\lambda$-lemma and nonexistence of invariant line fields (see {\cite[Theorem~1.1]{rvs}}), we show that roots of the~equation are isolated, hence there are only countably many parameters for which all critical values are prepoles.

\section{Dynamics of functions from families $\W_t$ and $\W_s$}

Let us collect some information about the families $\W_t$ and $\W_s$ which will be helpful in our proofs, for more we refer to~\cite{hk}. First recall that any elliptic function has no asymptotic values so the~postsingular set~$\PS(f_{\lambda})$ is the closure of the~critical trajectories. Moreover, the~Fatou set of any Weierstrass elliptic function contains no wandering domains, Baker domains or Herman rings (see {\cite[Lemma~5.2, Theorem~5.4]{hk}}).

Take any function $f_{\lambda}\in\W_t$. It has three critical values $e_1$, $e_2$ and $e_3$, all with the same modulus and forming the angle $2\pi{i}/3$ with each other, i.e. $e_2=\e^{2\pi{i}/3}e_1$ and $e_3=\e^{4\pi{i}/3}e_1$. Recall that the~triangular lattice is invariant under rotation by the angle $2\pi{i}/3$, thus the homogenity properties (cf. (3) in~\cite{hk}) gives that the~same relationship holds for every iterate of critical values, i.e. $f_{\lambda}^n(e_2)=\e^{2\pi{i}/3}f_{\lambda}^n(e_1)$ and $f_{\lambda}^n(e_3)=\e^{4\pi{i}/3}f_{\lambda}^n(e_1)$. Moreover, for any $n\geq0$ the derivative $f_{\lambda}'(f_{\lambda}^n(e_i))$ is the same for $i=1,2,3$. As a~consequence we obtain the following result (see {\cite[Proposition~5.3]{hk}}).
\begin{lemma}\label{Wtcases}
  For any function $f_{\lambda}\in\W_t$ one of the following occurs:
  	\begin{enumerate}
    	\item $J(f_{\lambda})=\Cd$;
    	\item For some perion $n$ and multiplier $0\leq\beta\leq1$ there exist exactly three (super) attracting or parabolic periodic cycles in $F(f_{\lambda})$ of period~$n$ with multiplier~$\beta$;
    	\item There exists exactly one (super) attracting or parabolic periodic cycle in $F(f_{\lambda})$ which contains all three critical values;
    	\item The only Fatou cycles are Siegel discs.
  	\end{enumerate}
\end{lemma}
Since the dynamics of all three critical values is basically the same, it is enough to know one of them to determine the~other two. In particular, if the~assumptions of Theorem~\ref{A} are satisfied, then necessarily every $e_i$ is not a~prepole and has a~bounded trajectory in $J(f_{\lambda})$ separated from $\Crit(f_{\lambda})$. On the~other hand, if one critical value is a~prepole, so are the~other two.

Passing to square lattices, take some $f_{\lambda}\in\W_s$. In this case we have the following critical values: $e_1$, $e_2=-e_1$ and $e_3=0$, which is a~pole of~$f_{\lambda}$, so the situation is even more rigid than before. By the definition $f_{\lambda}$ is even, so $e_1$ and $e_2$ share the~same trajectory which actually determines the~dynamics of $f_{\lambda}$ since $e_3$ is always a~pole. Thus, there are only three cases that may occur (see {\cite[Proposition~5.4]{hk}}).
\begin{lemma}\label{Wscases}
  For any function $f_{\lambda}\in\W_s$ one must occur:
  	\begin{enumerate}
    	\item $J(f_{\lambda})=\Cd$;
    	\item There exists exactly one (super) attracting or parabolic periodic cycle in $F(f_{\lambda})$;
    	\item The only Fatou cycles are Siegel discs.
  	\end{enumerate}
\end{lemma}
Now, if the~assumptions of Theorem~\ref{A} are satisfied, then all critical values are in $J(f_{\lambda})$, moreover, the trajectory of $e_1$ and $e_2$, which are not prepoles in this case, is bounded and separated from $\Crit(f_{\lambda})$. And similarly as for triangle lattices, if $e_1$ or $e_2$ is a~prepole, then all critical values of~$f_{\lambda}$ are prepoles.

As we mentioned at the~beginning there are various definitions of Misiurewicz condition in the~complex case. One of the~classical definitions, referred sometimes as \textit{pure Misiurewicz}, demands that every singular value is preperiodic, i.e. is eventually mapped onto a~repelling periodic cycle in the~Julia set. This condition, however, is very restrictive and we usually introduce more general definitions (very often depending on the~family of functions under consideration). In our case Definition~\ref{defMis} was inspired by the~close relation between critical trajectories of functions from families $\W_t$ and $\W_s$.

\section{Proof of the~Theorem~\ref{A}}

Denote by $\M$ the set of parameters satisfying the assumptions of Theorem~\ref{A} and by $e_{\lambda}\in{J(f_{\lambda})}$ the critical value of $f_{\lambda}$ (which is not a~prepole) with bounded trajectory not accumulating on~$\Crit(f_{\lambda})$. It follows that for every $\lambda\in\M$, we can find some $\delta>0$ such that
\beq\label{deltaMis} 
	\overline{O_{\lambda}(e_{\lambda})}\cap \Big(B\big(\Crit(f_{\lambda}),\delta\big)\cup B\big(\infty,\delta\big)\Big)= \emptyset, 
\eeq
where $O_{\lambda}(e_{\lambda})=\bigcup_{n\geq1}f^{n}_{\lambda}(e_{\lambda})$ is the forward trajectory of the critical value~$e_{\lambda}$ and balls are taken with respect to the~spherical metric. The set of parameters for which (\ref{deltaMis}) holds for any critical value $e_{\lambda}\in{J(f_{\lambda})}$ of~$f_{\lambda}$ will be denoted by $\M_{\delta}$.
Note that
\bdm \M=\bigcup_{n\geq1}\M_{1/n}\quad \textrm{and} \quad \delta_1<\delta_2 \Rightarrow \M_{\delta_1}\supset\M_{\delta_2}. \edm

Similarly to the case of the exponential family (cf.~\cite{b}) we will show, following Aspenberg's idea in \cite{a}, that parameters from $\M_{\delta}$ are rare in any neighbourhood of $\lambda_0\in\M$.

\begin{theorem}\label{B}
For families $\W_t$ and $\W_s$, if $\lambda_0\in\M$, then for every $\delta>0$, the set $\M_{\delta}$ has the~Lebesgue density strictly smaller than one at $\lambda_0$.
\end{theorem}

Obviously Theorem~\ref{B} implies that $\mu(\M_{\delta})=0$ for every $\delta>0$, where $\mu$ is the~Lebesgue measure on~$\Cd$. Hence
\bdm \mu(\M)\leq\sum_{n\geq1}\mu(\M_{1/n})=0, \edm
which is exactly the statement of Theorem~\ref{A}.

\medskip
In order to prove the~Theorem~\ref{B} we will focus on a~parameter $\lambda_0\in\M$ and its neighbourhood $B(\lambda_0,r)$ in the parameter plane. We will see how the~assumptions on the~critical value $e_{\lambda_0}$ and dynamical properties of families $\W_t$ and $\W_s$ imply exponential expansion on $\HS$, the~closure of the~forward trajectory of $e_{\lambda_0}$ under~$f_{\lambda_0}$. This leads to the existence of a~holomorphic motion $h:\HS\times B(\lambda_0,r)\to\C$ conjugating the dynamics of~$f_{\lambda_0}$ and nearby maps~$f_{\lambda}$, $\lambda\in{B(\lambda_0,r)}$, on a~neighbourhood of~$\HS$. Next, we will use the expansion property and the~absence of line fields for Misiurewicz elliptic maps to derive nice distortion properties binding space and parameter derivatives in a~small scale. This allows us to control the growth of a~parameter ball $B(\lambda_0,r)$ to a~big scale where in turn we can estimate the measure of those parameters which cannot belong to $\M_{\delta}\subset\M$.

\subsection{Holomorphic motion}\label{secholmot}

Take now a~parameter $\lambda_0\in\M$ for any of those two families. As we have just seen, all critical values of $f_{\lambda_0}$ are in the Julia set $J(f_{\lambda_0})$. Recall that the~Fatou set $F(f_{\lambda_0})$ has no wandering domains, Baker domains or Herman rings. Moreover, as we will see in a~moment, $f_{\lambda_0}$ is expanding on the~closure of a~critical trajectory and hence the~close relationship between trajectories of all critical values excludes existence of Siegel discs. We conclude that the Fatous set must be empty, thus $J(f_{\lambda_0})=\Cd$. Pick now one of the critical values in $J(f_{\lambda_0})$ which is not a~pole and denote it by $e_{\lambda}$. Here and in the following sections we use the spherical metric and derivatives unless otherwise stated. 

Consider the~set $\HS=\overline{O_{\lambda_0}(e_{\lambda_0})}$, the~closure of the~forward trajectory of~$e_{\lambda_0}$ under~$f_{\lambda_0}$. It is compact, forward invariant, contains neither critical nor parabolic points. Hence, by Theorem~1.2 in \cite{rvs} (compare also with {\cite[Theorem~1]{gks}}), $\HS$~is a~hyperbolic set, i.e. there are real constants $C>0$ and $a>1$ such that
\bdm |(f^n_{\lambda_0})'(z)|\geq Ca^n \;\textrm{ for all } z\in\HS \; \textrm{ and } n\geq1. \edm

Look now at the nearby maps $f_{\lambda}$, $\lambda\in B(\lambda_0,r)$ either in~$\W_t$ or in~$\W_s$. We will follow the proof of  {\cite[Theorem~III.1.6]{ms}} locally in a~neighbourhood of the hyperbolic set $\HS$ to show that if $r>0$ is sufficiently small, there exists a~holomorphic motion 
\bdm h\colon\HS\times B(\lambda_0,r)\to\C \edm
such that $h_{\lambda_0}=\mathrm{id}$, the map $h_{\lambda}:=h(\cdot,\lambda)\colon\HS\to\HS_{\lambda}$ is quasiconformal for each~$\lambda\in B(\lambda_0,r)$ and $h(z,\cdot)\colon{B(\lambda_0,r)}\to\C$ is holomorphic at every $z\in\HS$. Moreover, it respects the~dynamics, i.e.
\bdm h_{\lambda}\circ f_{\lambda_0}=f_{\lambda}\circ h_{\lambda} \;\textrm{ on }\; \HS. \edm

Notice first that $\HS$ contains no prepoles of~$f_{\lambda_0}$. Fix an $N\in\N$ such that
$$ \forall\;z\in\HS, \quad |(f_{\lambda_0}^N)'(z)|\geq2\tilde{a} $$
for some constant $\tilde{a}\gg1$. Take now a~neighbourhood $\NN$ of $\HS$ such that even in a~bigger neighbourhood $\NN_{\varepsilon}=B(\NN,\varepsilon)$, for some $\varepsilon>0$, there are neither critical points of~$f_{\lambda_0}$ nor prepoles of~$f_{\lambda_0}$ of orders $1,2,\ldots,N$.

Now, we want to choose small enough radius $r>1$ in the~parameter space. We do it in two steps, decreasing $\NN$ if necessary, so that the~following two conditions are satisfied:
\begin{enumerate}
	\item $\forall\;\lambda\in{B(\lambda_0,r)}$, the~set $\NN$ contains neither critical points nor prepoles of $f_{\lambda}$ of orders $1,2,\ldots,N$.
	\item $\forall\;\lambda\in{B(\lambda_0,r)}, \; \forall\;z\in\NN, \quad |(f_{\lambda}^N)'(z)|\geq\tilde{a}\gg1$.
\end{enumerate}
It is possible since critical points and poles depend analytically on the~parameter~$\lambda$ and the~derivative $(f_{\lambda}^N)'(z)$ changes continuously with~$\lambda$.

The~choice of $r>0$ guarantees the~expanding property for all functions $f_{\lambda}$, $\lambda\in{B(\lambda_0,r)}$, where the~constants $C>0$ and $a>1$ may have changed. 
\begin{lemma}\label{expand} There are constants $C>0$, $a>1$ and a~radius $r>0$ such that whenever $f^j_{\lambda}(z)\in\NN$ for $j=0,\ldots,k$ and $\lambda\in{B(\lambda_0,r)}$, then
\bdm |(f^k_{\lambda})'(z)|\geq Ca^k. \edm \end{lemma}

Next step is to introduce an~appropriate adapted metric defined for $z\in\NN$ as follows
$$ d(z)=\frac{1}{N}\sum_{n=0}^{N-1}|(f_{\lambda_0}^n)'(z)|. $$
By the~careful choice of~$\NN$ we get that $d(z)\leq{C_1}$ for all $z\in\NN$. Additionally, we can modify $C_1$ so that the~estimate remains valid for every function $f_{\lambda}$, $\lambda\in{B(\lambda_0,r)}$, decreasing $r$ if necessary.

Let us compute derivative~$|f'|_d$ of the~function~$f:=f_{\lambda_0}$ with respect to the~adapted metric for $z\in\NN$.
$$\aligned 
|f'(z)|_d=&|f'(z)|\frac{d(f(z))}{d(z)}= \frac{|f'(z)|\frac{1}{N}\sum\limits_{n=0}^{N-1}|(f^n)'(f(z))|}{\frac{1}{N}\sum\limits_{n=0}^{N-1}|(f^n)'(z)|}= \frac{\frac{1}{N}\sum\limits_{n=0}^{N-1}|(f^{n+1})'(z)|}{\frac{1}{N}\sum\limits_{n=0}^{N-1}|(f^n)'(z)|}= \\
=&1+\frac{\frac{1}{N}(|(f^N)'(z)|-1)}{\frac{1}{N}\sum\limits_{n=0}^{N-1}|(f^n)'(z)|}\geq 1+\frac{\tilde{a}-1}{NC_1}>1, 
\endaligned$$
hence $|(f_{\lambda_0})'|_d\geq const>1$ on $\NN$.

Take now a~nearby function $g:=f_{\lambda}$, where $\lambda\in{B(\lambda_0,r)}$ for sufficiently small $r>0$, and $z\in\NN$.
$$\aligned
|g'(z)|_d=&|g'(z)|\frac{d(g(z))}{d(z)}= \frac{|g'(z)|\frac{1}{N}\sum\limits_{n=0}^{N-1}|(f^n)'(g(z))|}{\frac{1}{N}\sum\limits_{n=0}^{N-1}|(f^n)'(z)|}= \frac{|g'(z)|}{|f'(z)|}\frac{\frac{1}{N}\sum\limits_{n=0}^{N-1}|(f^n\circ{g})'(z)|}{\frac{1}{N}\sum\limits_{n=0}^{N-1}|(f^{n+1})'(z)|}|f'(z)|_d.
\endaligned$$
Since $|(f_{\lambda_0})'(z)|_d\geq const>1$ on~$\NN$, therefore if the~radius $r>0$ is sufficiently small (decreasing $\NN$ if necessary), then for any $\lambda\in{B(\lambda_0,r)}$, 
$$ |(f_{\lambda})'|_d\geq\tilde{C}>1 \quad\textrm{on} \quad \NN.$$
This is a~consequence of the~form of derivative with respect to the~adapted metric as we consider only finitely many iterates, there are no prepoles of~$f_{\lambda}$ of orders $1,2,\ldots,N$ in~$\NN$ and values of functions and iterates (which are holomorphic, bounded and equicontinuous on~$\NN$) depend continuously on~$\lambda$.

We proceed exactly as in~\cite{ms}. Let $\varepsilon>0$ be such that for every $z\in\HS$, $B(z,\varepsilon)_d\subset\NN$ (the~ball with respect to the~adapted metric). If the~radius $r>0$ is sufficiently small, then for every $\lambda\in{B(\lambda_0,r)}$ we have $f_{\lambda}(B(z,\varepsilon)_d)\supset B(f_{\lambda_0}(z),\varepsilon)_d$. Hence for every $n\in\N$ and $z\in\HS$, the~set
$$ W_{\lambda,n}=\left\{w: f_{\lambda}^k(w)\in{B\left(f_{\lambda_0}^k(z),\varepsilon\right)_d \;\textrm{for}\; k=0,1,\ldots,n}\right\} $$
is nonempty and its diameter does not exceed $2\varepsilon\tilde{C}^{-n}$. There exists, therefore, a~unique point $h_{\lambda}(z)$ such that $f_{\lambda}^n(h_{\lambda}(z))\in{B(f_{\lambda_0}^n(z),\varepsilon)_d}$ for all $n\in\N$. We get immediately that $h_{\lambda}(f_{\lambda_0}(z))=f_{\lambda}(h_{\lambda}(z))$. Moreover, $h_{\lambda}$ is continuous and injective.

Since the~holomorphic motion~$h\colon\HS\times B(\lambda_0,r)\to\C$ respects the~dynamics and $f_{\lambda_0}(\HS)\subset\HS$ we immediately get that
$$ f_{\lambda}(h_{\lambda}(\HS))=h_{\lambda}(f_{\lambda_0}(\HS))\subset h_{\lambda}(\HS), $$
thus the~set $\HS_{\lambda}:=h_{\lambda}(\HS)$ is $f_{\lambda}$-invariant and by the~Lemma~\ref{expand}, it is a~hyperbolic set for $f_{\lambda}$.

\medskip
Now, we want to obtain so-called transversality condition (cf. \cite{a}), which says that the critical value~$e_{\lambda}$ of $f_{\lambda}$ cannot follow the holomorphic motion $h_{\lambda}(e_{\lambda_0})$ of the critical value of $f_{\lambda_0}$ in the whole parameter ball $B(\lambda_0,r)$. In the triangular case it follows e.g. from the non-existence of invariant line-fields for Misiurewicz maps proved by Graczyk, Kotus and \'Swi\k{a}tek in {\cite[Theorem~2]{gks}}, for the case of square lattices we refer to the~more general result {\cite[Theorem~1.1]{rvs}}. For the convenience of the reader, we will use notation analogous to \cite{a}.

Recall that there is a~strong relationship between the trajectories of critical values of functions in both families $\W_t$ and $\W_s$, in particular the trajectory of $e_{\lambda}$ determines the dynamics of~$f_{\lambda}$. Consider a~holomorphic function $x:B(\lambda_0,r)\to\C$ given by
\bdm x(\lambda)=e_{\lambda}-h_{\lambda}(e_{\lambda_0}) \edm
which is exactly the difference between the~critical value of $f_{\lambda}$ and the~holomorphic motion of the~critical value of the~starting map $f_{\lambda_0}$ (we assume that the~radius of the~parameter ball is so small that there is only one critical value of~$f_{\lambda}$ close to~$e_{\lambda_0}$). Note that $h_{\lambda}(e_{\lambda_0})$ always belongs to the~hyperbolic set $\HS_{\lambda}$. We obviously have that $x(\lambda_0)=0$. Our aim is to show that $\lambda_0$ is an~isolated zero of $x$.

\begin{lemma}\label{transver} The function $x$ is not identically zero in any ball $B(\lambda_0,r)$ in the parameter plane. \end{lemma}
\begin{proof} 
Suppose that $x(\lambda)\equiv0$ on some ball $B(\lambda_0,r)$ which means that for any $\lambda$ close to~$\lambda_0$, the~trajectory of the~critical value $e_{\lambda}$ stays in the~appropriate hyperbolic set $\HS_{\lambda}$. It follows that the~trajectories of all critical values of~$f_{\lambda}$, except for the~pole $e_3$ in the case of square lattice, lie in some hyperbolic set. Thus, the~parameter~$\lambda_0$ is postsingularly stable since trajectories of all critical values of~$f_{\lambda}$ behave the~same for all parameters~$\lambda$ close to~$\lambda_0$. We can, therefore, extend $h_{\lambda}$ to a~quasiconformal conjugacy on the~consecutive preimages of~$e_{\lambda}$ and next, by the~$\lambda$-Lemma (cf. {\cite[$\lambda$-Lemma]{mss}}), to a~quasiconformal conjugacy on~the whole Julia set $J(f_{\lambda_0})=\Cd$ between $f_{\lambda_0}$ and $f_{\lambda}$ for any $\lambda\in B(\lambda_0,r)$. In this case however, there would be an~$f_{\lambda_0}$-invariant line field on~$J(f_{\lambda_0})$ which cannot exist by {\cite[Theorem~1.1]{rvs}} (cf. {\cite[Theorem~2]{gks}}).  
\end{proof}

Therefore we have that
\beq\label{xform} x(\lambda)=\alpha_K(\lambda-\lambda_0)^{K}+\alpha_{K+1}(\lambda-\lambda_0)^{K+1}+\ldots  \eeq
for some $K\geq1$ and $\alpha_K\neq0$. This property will be crucial to obtain distortion estimates in the~next section.

\subsection{Distortion estimates}

In this section we derive distortion estimates based on the expansion property near the hyperbolic set~$\HS$. It is rather technical and mainly follows analogous proofs in \cite{a} and~\cite{b}. We decided however to keep it in a~very detailed form for the convenience of the reader and also because of changes which are minor but crucial.

Recall that we have chosen the~neighbourhood $\NN$ of the~hyperbolic set~$\HS$ and the~radius $r>0$ so that for all functions $f_{\lambda}$, $\lambda\in{B(\lambda_0,r)}$, we have the~expansion property stated in Lemma~\ref{expand}. Assume moreover that $\NN$ is closed, bounded (hence compact in~$\C$) and for some $\delta>0$,
$$  \NN\cap\big(B\big(\Crit(f_{\lambda}),\delta\big)\cup{B\big(\infty,\delta\big)\big)}=\emptyset. $$

If we now take some $\delta'>0$ for which $\{z:\dist(z,\HS)\leq11\delta'\}\subset\NN$, then we will always assume $r>0$ to be so small that $\{z:\dist(z,\HS_{\lambda})\leq10\delta'\}\subset\NN$ for each $\lambda\in B(\lambda_0,r)$. This means that $\HS_{\lambda}$, the~hyperbolic set for $f_{\lambda}$, is well inside $\NN$.

The~neighbourhood $\NN$ was chosen so that for some $N\geq1$, $\tilde{a}>1$ and for all $z\in\NN$, $\lambda\in B(\lambda_0,r)$, we have $|(f^N_{\lambda})'(z)|\geq\tilde{a}$. Thus for every $z\in\NN$ we can find some radius $r(z)>0$ such that 
\beq\label{expansion} |f^N_{\lambda}(z)-f^N_{\lambda}(w)|\geq\tilde{a}|z-w| \eeq
for all $w\in\NN$ with $|z-w|\leq r(z)$ (decreasing slightly $\tilde{a}>1$ if necessarily). Since $\NN$ is compact and $r(z)$ changes continuously, we can find a~universal $\tilde{r}>1$ such that (\ref{expansion}) holds for every $z,w\in\NN$ with $|z-w|\leq\tilde{r}$. This implies exponential expansion in a~small scale.

\begin{lemma}\label{expansionlemma} 
There are constants $\tilde{\delta},C>0$ and $a>1$ such that for every $\lambda\in{B(\lambda_0,r)}$ and every $z,w\in\NN$, if 
$f^j_{\lambda}(z),f^j_{\lambda}(w)\in\NN$ and $|f^j_{\lambda}(z)-f^j_{\lambda}(w)|\leq\tilde{\delta}$ for $j=0,\ldots,k$, then 
\bdm  |f^k_{\lambda}(z)-f^k_{\lambda}(w)|\geq Ca^k |z-w|. \edm 
\end{lemma}
\begin{proof} 
Every integer $k$ can be written in the form $k=pN+q$, where $q\leq N-1$. For some $\tilde{C},\tilde{\delta}>0$ we can estimate for all $\lambda\in B(\lambda_0,r)$
\bdm |f_{\lambda}(z)-f_{\lambda}(w)|\geq\tilde{C}|z-w| \; \textrm{ for all } z,w\in\NN \; \textrm{ with } |z-w|\leq\tilde{\delta}. \edm
If we now take $z,w\in\NN$ for which assumptions of the lemma are satisfied, then
\bdm |f^k_{\lambda}(z)-f^k_{\lambda}(w)|\geq\tilde{a}^p|f^q_{\lambda}(z)-f^q_{\lambda}(w)|\geq \tilde{a}^p\tilde{C}^q|z-w| \geq a^kC|z-w| \edm
for $a=\tilde{a}^{\frac{1}{m}}$ and some $C>0$.
\end{proof}

We will use the expansion property in the following distortion estimates to show that in a~small scale parameter and space derivatives are comparable. For $\lambda\in B(\lambda_0,r)$ and $n\geq0$ put
\bdm \xi_n(\lambda)=f^n_{\lambda}(e_{\lambda}) \quad \textrm{ and } \quad \mu_n(\lambda)=f^n_{\lambda}(h_{\lambda}(e_{\lambda_0}))=h_{\lambda}(f^n_{\lambda_0}(e_{\lambda_0})). \edm
Then $\xi_n(\lambda)$ is the~forward orbit of the~critical value for $f_{\lambda}$ while $\mu_n(\lambda)$ is the~holomorphic motion of the~critical orbit for $f_{\lambda_0}$, hence $\mu_n(\lambda)\in\HS_{\lambda}$. In particular $x(\lambda)=\xi_0(\lambda)-\mu_0(\lambda)$.

The following lemma will be used several times in our distortion estimates. See \cite{a} for references.
\begin{lemma}\label{helpful} 
Let $u_n\in\C$ for $n=1,\ldots,N$. Then
\bdm \left|\prod_{n=1}^{N}(1+u_n)-1\right| \leq \exp\left(\sum_{n=1}^{N}|u_n|\right)-1. \edm 
\end{lemma}

Let us begin with the Main Distortion Lemma concerning control of the space derivative in a~neighbourhood of the hyperbolic set.
\begin{lemma}\label{mdl} 
For every $\varepsilon>0$ we can find $\delta'>0$ and $r>0$ arbitrarily small with the following property. For any $a,b\in{B(\lambda_0,r)}$ if $|\xi_k(\lambda)-\mu_k(\lambda)|\leq\delta'$ for all $k\leq{n}$ and $\lambda=a,b$, then
\bdm   \left|\frac{(f^n_a)'(e_a)}{(f^n_b)'(e_b)}-1\right|<\varepsilon. \edm 
\end{lemma}
\begin{proof} 
First we will show that for an~arbitrarily small $\varepsilon_1=\varepsilon_1(\delta')$, it is possible to choose $\delta'>0$ so that 
\beq\label{dist1} \left|\frac{(f^n_{\lambda})'(\mu_0(\lambda))}{(f^n_{\lambda})'(\xi_0(\lambda))}-1\right|\leq \varepsilon_1 \eeq  
provided $|\xi_k(\lambda)-\mu_k(\lambda)|\leq\delta'$ for all $k\leq{n}$.

By the expansion property and since $|f_{\lambda}'|>C_{\delta}^{-1}$ on $\NN$ for some $C_{\delta}>0$, we can estimate for any $\lambda\in{B(\lambda_0,r)}$:
\bdm\aligned
 \sum_{j=0}^{n-1}&\left|\frac{f_{\lambda}'(\mu_j(\lambda))-f_{\lambda}'(\xi_j(\lambda))}{f_{\lambda}'(\xi_j(\lambda))}\right|\leq C_{\delta}\sum_{j=0}^{n-1}\left|f_{\lambda}'(\mu_j(\lambda))-f_{\lambda}'(\xi_j(\lambda))\right| \leq \\ 
 &\leq C_{\delta}\max_{z\in\NN}|f_{\lambda}''(z)|\sum_{j=0}^{n-1}|\mu_j(\lambda)-\xi_j(\lambda)| \leq \tilde{C}\sum_{j=0}^{n-1}Ca^{j-n}|\mu_n(\lambda)-\xi_n(\lambda)|\leq C'\delta',  
\endaligned\edm
where $\max|f_{\lambda}''(z)|$ is bounded on $B(\lambda_0,r)$ since $\NN$ contains no poles of~$f_{\lambda}^j$ for $j=1,\ldots,N$ and $\lambda\in{B(\lambda_0,r)}$. Using Lemma~\ref{helpful} we obtain the~inequality (\ref{dist1}) if $\delta'>0$ is small enough.

Secondly, for any $\varepsilon_2>0$, if $\delta'>0$ and $r>0$ are chosen sufficiently small, then for every $t,s\in{B(\lambda_0,r)}$,
\beq\label{dist2} \left|\frac{(f^n_t)'(\mu_0(t))}{(f^n_s)'(\mu_0(s))}-1\right|\leq\varepsilon_2. \eeq 
Put $a_{\lambda,j}=f_{\lambda}'(\mu_j(\lambda))$. Since each $a_{\lambda,j}$ is analytic with respect to~$\lambda$, it can be expressed as follows: $a_{\lambda,j}=a_{\lambda_0,j}(1+c_j(\lambda-\lambda_0)^l+\ldots)$. Moreover, by Lemma~\ref{expansionlemma} and (\ref{xform}), we have that
\beq\label{nlog} n\leq-C\log|x(\lambda)|\leq-\tilde{C}\log|\lambda-\lambda_0|,\eeq 
where constants depend only on~$\delta'$ and not on $n$. Thus, if $c=\sum_{j=0}^{n-1}c_j$, then
\bdm \frac{(f^n_t)'(\mu_0(t))}{(f^n_s)'(\mu_0(s))}=\prod_{j=0}^{n-1}\frac{a_{t,j}}{a_{s,j}}= \prod_{j=0}^{n-1}\frac{a_{\lambda_0,j}(1+c_j(t-\lambda_0)^l+\ldots)}{a_{\lambda_0,j}(1+c_j(s-\lambda_0)^l+\ldots)}= \frac{1+cn(t-\lambda_0)^l+\ldots}{1+cn(s-\lambda_0)^l+\ldots}. \edm
Now, both the~numerator and the~denominator can be made arbitrarily close to one if only $r>0$ is small enough, since they are of order $1+\O(|t-\lambda_0|^l\log|t-\lambda_0|)$ and $1+\O(|s-\lambda_0|^l\log|s-\lambda_0|)$.

Putting together (\ref{dist1}) and (\ref{dist2}) we obtain the statement of the lemma. 
\end{proof}

\vspace{1em}
Next we want to compare space and parameter derivatives.
\begin{lemma}\label{compare} 
Let $\varepsilon>0$. If $\delta'>0$ is sufficiently small, then for every $0<\delta''<\delta'$, there exists an~$r>0$ such that the following holds. For any $\lambda\in{B(\lambda_0,r)}$, if $|\xi_k(\lambda)-\mu_k(\lambda)|\leq\delta'$ for $k\leq{n}$ and $|\xi_n(\lambda)-\mu_n(\lambda)|\geq\delta''$, then
\bdm \left|\frac{\xi_n'(\lambda)}{(f^n_{\lambda})'(\mu_0(\lambda))x'(\lambda)}-1\right|\leq\varepsilon. \edm 
\end{lemma}
\begin{proof}
Note that we have
\beq\label{Eestim} \xi_n(\lambda)=\mu_n(\lambda)+(f^n_{\lambda})'(\mu_0(\lambda))x(\lambda)+E_n(\lambda),  \eeq 
where $|E_n(\lambda)|\leq\varepsilon_1|\xi_n(\lambda)-\mu_n(\lambda)|$ independently of $n$, for any small $\varepsilon_1>0$, if only $\delta'>0$ was chosen small enough. To see this we will proceed similarly as in the first part of the proof of Lemma~\ref{mdl}. First we can write
\bdm  \frac{(f^n_{\lambda})'(\mu_0(\lambda))x(\lambda)}{\xi_n(\lambda)-\mu_n(\lambda)}= \prod_{j=0}^{n-1}\frac{f_{\lambda}'(\mu_j(\lambda))(\xi_j(\lambda)-\mu_j(\lambda))}{\xi_{j+1}(\lambda)-\mu_{j+1}(\lambda)}. \edm
By the expansion property (Lemma~\ref{expansionlemma}) we can estimate as follows
\bdm\aligned  
\left|\frac{f_{\lambda}'(\mu_j(\lambda))(\xi_j(\lambda)-\mu_j(\lambda))}{\xi_{j+1}(\lambda)-\mu_{j+1}(\lambda)}-1\right|\leq  \frac{1}{Ca}&\left|f_{\lambda}'(\mu_j(\lambda))-\frac{\xi_{j+1}(\lambda)-\mu_{j+1}(\lambda)}{\xi_j(\lambda)-\mu_j(\lambda)} \right| \leq \\ \leq\frac{1}{Ca}\max_{z\in\NN}|f_{\lambda}''(z)|\;|\xi_j(\lambda)-\mu_j(\lambda)| &\leq \frac{M''}{Ca}C^{-1}a^{j-n}|\xi_n(\lambda)-\mu_n(\lambda)|,
\endaligned\edm 
for $M''=\max\{|f_{\lambda}''(z)|:z\in\NN,\lambda\in{B(\lambda_0,r)}\}$, which is finite by our careful choice of~$\NN$. Applying Lemma~\ref{helpful} we obtain the estimate we were looking for. 

Put again $f_{\lambda}'(\mu_j(\lambda))=a_{\lambda,j}$, then $(f^n_{\lambda})'(\mu_0(\lambda))=\prod_{j=0}^{n-1}a_{\lambda,j}$. Now, differentiate~$\xi_n$ with respect to $\lambda$. By the Chain Rule we get
\bdm\aligned 
\xi_n'(\lambda)= &\mu_n'(\lambda)+x'(\lambda)\prod_{j=0}^{n-1}a_{\lambda,j}+x(\lambda)\sum_{j=0}^{n-1}a_{\lambda,j}'\frac{\prod_{k=0}^{n-1}a_{\lambda,k}}{a_{\lambda,j}}+E_n'(\lambda)= \\
=&\prod_{j=0}^{n-1}a_{\lambda,j}\left(x'(\lambda)+x(\lambda)\sum_{j=0}^{n-1}\frac{a_{\lambda,j}'}{a_{\lambda,j}}+ \frac{\mu_n'(\lambda)+E_n'(\lambda)}{\prod_{j=0}^{n-1}a_{\lambda,j}}\right). 
\endaligned\edm
In the following we want to show that $x'(\lambda)$ is the~leading term in the above expression. 

Recall that $\delta''\leq|\xi_n(\lambda)-\mu_n(\lambda)|\leq\delta'$, thus by (\ref{Eestim}) and the~estimate on~$|E_n(\lambda)|$ we have
\beq\label{both} (1-\varepsilon_1)\delta''\leq |x(\lambda)|\prod_{j=0}^{n-1}|a_{\lambda,j}|\leq (1+\varepsilon_1)\delta' \eeq

Now we need to estimate $|\sum\frac{a_{\lambda,j}'}{a_{\lambda,j}}|$. Note that, since $\mu_j(\lambda)=f^j_{\lambda}(\mu_0(\lambda))\in\HS_{\lambda}$, we get that
\bdm |a_{\lambda,j}|=|f_{\lambda}'(\mu_j(\lambda))|\leq \max_{z\in\HS_{\lambda},\lambda\in{B(\lambda_0,r)}}|f_{\lambda}'(z)| \quad\textrm{and}\quad |a_{\lambda,j}|\geq{Ca}\,,\; C,a>0.\edm
Since $a_{\lambda,j}$ are uniformly bounded for every $j$ and $\lambda\in{B(\lambda_0,r)}$, therefore, by Cauchy's formula, also $a_{\lambda,j}'$ are uniformly bounded by some $M'>0$ on a~slightly smaller ball $B(\lambda_0,r')$. We get the following
\bdm \left|\sum_{j=0}^{n-1}\frac{a_{\lambda,j}'}{a_{\lambda,j}}\right|\leq \sum_{j=0}^{n-1}\left|\frac{a_{\lambda,j}'}{a_{\lambda,j}}\right|\leq n\frac{M'}{Ca}=:n\tilde{C}. \edm
Thus, using (\ref{nlog}),
\bdm |x(\lambda)|\left|\sum\frac{a_{\lambda,j}'}{a_{\lambda,j}}\right|\leq |x(\lambda)|n\tilde{C}\leq |x(\lambda)|C'(-\log|x(\lambda)|)\tilde{C}, \edm
where $C'>0$ depends only on $\delta'$. Moreover, up to a~multiplicative constant,
\beq\label{summ1} \frac{-|x(\lambda)|\log|x(\lambda)|}{|x'(\lambda)|}\asymp \frac{-|(\lambda-\lambda_0)^K|\log|\lambda-\lambda_0|}{|(\lambda-\lambda_0)^{K-1}|} \asymp -|\lambda-\lambda_0|\log|\lambda-\lambda_0|. \eeq

Let us estimate
\bdm\aligned
\frac{\xi_n'(\lambda)}{(f^n_{\lambda})'(\mu_0(\lambda))x'(\lambda)}-1=& \frac{\prod{a_{\lambda,j}\left(x'(\lambda)+x(\lambda)\sum\frac{a_{\lambda,j}'}{a_{\lambda,j}}+\frac{\mu_n'(\lambda)+E_n'(\lambda)}{\prod{a_{\lambda,j}}}\right)}} {\prod{a_{\lambda,j}}\,x'(\lambda)}-1=\\
=&\frac{x(\lambda)\sum\frac{a_{\lambda,j}'}{a_{\lambda,j}}}{x'(\lambda)} +\frac{\mu_n'(\lambda)+E_n'(\lambda)}{\prod{a_{\lambda,j}}\,x'(\lambda)}. 
\endaligned\edm
By (\ref{summ1}) the first summand tends uniformly to zero as $\lambda\to\lambda_0$. To see what happens with the second summand note that $|\mu_n'(\lambda)+E_n'(\lambda)|$ is uniformly bounded by Cauchy's formula, since $\mu_n(\lambda)$ and $E_n(\lambda)$ are bounded. We have also seen that $|\prod{a_{\lambda,j}}\,x(\lambda)|$ is bounded (from both sides) independently of $n$. Therefore, by (\ref{both}), we get
\bdm \left|\frac{1}{\prod{a_{\lambda,j}}\,x'(\lambda)}\right|= \left|\frac{1}{\prod{a_{\lambda,j}}\,x(\lambda)}\right|\left|\frac{x(\lambda)}{x'(\lambda)}\right|\leq \frac{1}{\delta''(1-\varepsilon_1)}\left|\frac{x(\lambda)}{x'(\lambda)}\right|\asymp|\lambda-\lambda_0|, \edm
thus also the second summand tends uniformly to zero as $\lambda\to\lambda_0$. This finishes the proof. 
\end{proof}

\vspace{1em}
Binding together Lemma~\ref{mdl} and Lemma~\ref{compare} we obtain the following result.
\begin{corollary}\label{distcor} 
Let $\varepsilon>0$. If $\delta'>0$ is small enough and $0<\delta''<\delta'$, we can find a~radius $r>0$ such that for every $\lambda\in{B(\lambda_0,r)}$ if $|\xi_k(\lambda)-\mu_k(\lambda)|\leq\delta'$ for $k\leq{n}$ and $|\xi_n(\lambda)-\mu_n(\lambda)|\geq\delta''$, then
\bdm \left|\frac{\xi_n'(\lambda)}{(f^n_{\lambda})'(e_{\lambda})\,x'(\lambda)}-1 \right|\leq\varepsilon. \edm 
\end{corollary}

\subsection{Distortion in an annulus}

As we have seen in the previous section, we need to move away from $\lambda_0$ in the parameter ball $B(\lambda_0,r)$ in order to have nice distortion estimates. That is why we will restrict our considerations to an annular domain. This approach will give us a~powerful tool which is bounded distortion of~$\xi_n$ and will lead to the control of the growth of $B(\lambda_0,r)$ under $\xi_n$.

Consider an~annulus in the parameter space:
\bdm A=A(\lambda_0;r_1,r_2)=\{\lambda:r_1<|\lambda-\lambda_0|<r_2\}. \edm
Note that, by (\ref{xform}), for some constant $C\geq1$ and any $\lambda_1,\lambda_2\in{A}$,
\bdm C^{-1}\left(\frac{r_1}{r_2}\right)^{K-1}\leq \left|\frac{x'(\lambda_1)}{x'(\lambda_2)}\right|\leq C\left(\frac{r_2}{r_1}\right)^{K-1}, \edm
where $K$ is the degree of $x(.)$ at $\lambda_0$. Therefore from Corollary~\ref{distcor} and Lemma~\ref{mdl} we conclude that if $r_2>0$ is small enough, then
\bdm \tilde{C}^{-1}\left(\frac{r_1}{r_2}\right)^{K-1}\leq \left|\frac{\xi_n'(\lambda_1)}{\xi_n'(\lambda_2)}\right|\leq \tilde{C}\left(\frac{r_2}{r_1}\right)^{K-1}, \edm
for some $\tilde{C}\geq1$ and all $\lambda_1,\lambda_2\in{A}$, as long as $|\xi_k(\lambda)-\mu_k(\lambda)|\leq\delta'$ for $k\leq{n}$ and $|\xi_n(\lambda)-\mu_n(\lambda)|\geq\delta''$ for all $\lambda\in{A}$.

\begin{lemma}\label{annulus} 
Let $\varepsilon>0$. If $\delta'>0$ and $\frac{\delta''}{\delta'}$ are sufficiently small, $0<\delta''<\delta'$, there exists an~$r>0$ such that for any ball $B=B(\lambda_0,r_2)\subset{B(\lambda_0,r)}$ we have the following. Let $n$ be maximal for which $|\xi_n(\lambda)-\mu_n(\lambda)|\leq\delta'$ for all $\lambda\in{B}$. Let $r_1<r_2$ be minimal such that $|\xi_n(\lambda)-\mu_n(\lambda)|\geq\delta''$ for all $\lambda\in{A}=A(\lambda_0;r_1,r_2)$. Then $\frac{r_1}{r_2}\leq\frac{1}{10}$ and there is some $\delta''<\delta_1'<\delta'$ such that
\bdm A(\mu_n(\lambda_0);\delta''+\varepsilon,\delta_1'-\varepsilon)\subset \xi_n(A)\subset A(\mu_n(\lambda_0);\delta''-\varepsilon,\delta_1'+\varepsilon). \edm
Moreover, $\xi_n$ is at most $K$-to-$1$ on $B$. 
\end{lemma}
\begin{proof}
Note that a~parameter circle $\gamma_r=\{\lambda:|\lambda-\lambda_0|=r\}$, for small $r>0$, is mapped under $x(.)$ onto a~curve that encircles $\lambda_0$ $K$-times so that $x(\gamma_r)$ is close to a circle of radius $\alpha_Kr^K$. Moreover, $|\mu_n(\lambda)-\mu_n(\lambda_0)|=|h_{\lambda}(f^n_{\lambda_0}(e_{\lambda_0}))-f^n_{\lambda_0}(e_{\lambda_0})|$ is arbitrarily small for small radii in the~parameter space, since $\HS$ and $\HS_{\lambda}$ can be very close to each other for $\lambda\in{B(\lambda_0,r)}$. Thus, if $r$ is small and $|\xi_n(\lambda)-\mu_n(\lambda)|\geq\delta''$, then
\beq\label{P} |\xi_n(\lambda)-\mu_n(\lambda)|>P|\mu_n(\lambda)-\mu_n(\lambda_0)| \eeq
for some big $P\gg1$ depending only on $\delta''$ and $r$. Arguing again like in the proof of Lemma~\ref{compare}, we get that for every $\varepsilon_1>0$ we can choose $\delta'>0$ and $r>0$ so that 
\beq\label{pom}  |\xi_n(\lambda)-\mu_n(\lambda)-(f^n_{\lambda})'(e_{\lambda})x(\lambda)|<\varepsilon_1|\xi_n(\lambda)-\mu_n(\lambda)| \eeq
for all $\lambda\in{B(\lambda_0,r)}$. 

If $r_1$ is minimal so that $|\xi_n(\lambda)-\mu_n(\lambda)|\geq\delta''$ for all $\lambda\in{A(\lambda_0;r_1,r_2)}$, then for some $\lambda_1$ with $|\lambda_1-\lambda_0|=r_1$ we have 
\beq\label{lambda1} |\xi_n(\lambda_1)-\mu_n(\lambda_1)|=\delta''.\eeq 
On the other hand, from the definition of $n$, we have for some $\lambda_2$ with $|\lambda_2-\lambda_0|=r_2$ that $|\xi_{n+1}(\lambda_2)-\mu_{n+1}(\lambda_2)|\geq\delta'$. But
\bdm |\xi_{n+1}(\lambda_2)-\mu_{n+1}(\lambda_2)|=|f_{\lambda_2}(\xi_n(\lambda_2))-f_{\lambda_2}(\mu_n(\lambda_2))|\leq M'|\xi_n(\lambda_2)-\mu_n(\lambda_2)|, \edm
where $M'=\max\{|f_{\lambda}'(z)|:z\in\NN,\lambda\in{B(\lambda_0,r)}\}$ which is finite since $\NN$ contains neither poles nor essential singularities of~$f_{\lambda}$. Therefore we get that 
\beq\label{lambda2} |\xi_n(\lambda_2)-\mu_n(\lambda_2)|\geq\frac{\delta'}{M'}.\eeq 
Moreover, by (\ref{pom}), for every $\lambda\in{B(\lambda_0,r)}$, if $r>0$ and $\delta'>0$ were small enough, then
\beq\label{pompom} \frac{1}{1+\varepsilon_1}|(f^n_{\lambda})'(e_{\lambda})x(\lambda)|\leq |\xi_n(\lambda)-\mu_n(\lambda)|\leq \frac{1}{1-\varepsilon_1}|(f^n_{\lambda})'(e_{\lambda})x(\lambda)|. \eeq
Using (\ref{lambda1}), (\ref{lambda2}), (\ref{pompom}) and Lemma~\ref{mdl} we can estimate as follows 
\bdm \frac{\delta'}{\delta''}\leq \frac{M'|\xi_n(\lambda_2)-\mu_n(\lambda_2)|}{|\xi_n(\lambda_1)-\mu_n(\lambda_1)|}\leq M'\frac{1+\varepsilon_1}{1-\varepsilon_1}\left|\frac{(f^n_{\lambda_2})'(e_{\lambda_2})x(\lambda_2)}{(f^n_{\lambda_1})'(e_{\lambda_1})x(\lambda_1)}\right|\leq M'\frac{(1+\varepsilon_1)^2}{1-\varepsilon_1}\left|\frac{x(\lambda_2)}{x(\lambda_1)}\right|. \edm
Thus we can choose $\delta''>0$ so small that $\frac{r_1}{r_2}\leq\frac{1}{10}$ independently of $n$.

Now we want to see how many times $\xi_n(\lambda)-\mu_n(\lambda)$ orbits around $0$, as the parameter $\lambda$ moves along the~circle $\gamma_r$, $r>r_1$. To see this let us look at the expression $\frac{\xi_n(\lambda)-\mu_n(\lambda)}{|\xi_n(\lambda)-\mu_n(\lambda)|}$. But by (\ref{pom}) we have that
\bdm \left|\frac{\xi_n(\lambda)-\mu_n(\lambda)}{|\xi_n(\lambda)-\mu_n(\lambda)|}- \frac{(f^n_{\lambda})'(e_{\lambda})x(\lambda)}{|\xi_n(\lambda)-\mu_n(\lambda)|}\right|\leq\varepsilon_1, \edm
so it is the same to ask how many times $(f^n_{\lambda})'(e_{\lambda})x(\lambda)$ encircles $0$. By Lemma~\ref{mdl}, $(f^n_{\lambda})'(e_{\lambda})$ is essentially constant on $B(\lambda_0,r_2)$, so the number we are looking for is $K$, the same as for $x(\lambda)$ only. Further, recall after (\ref{P}) that $|\mu_n(\lambda)-\mu_n(\lambda_0)|$ is much smaller than $|\xi_n(\lambda)-\mu_n(\lambda)|$. This means that $\xi_n(\lambda)$ orbits around $\mu_n(\lambda_0)=\xi_n(\lambda_0)$ also $K$ times close to some circle centered at~$\mu_n(\lambda_0)$. By the Argument Principle, the degree of~$\xi_n$ is at most~$K$.

In order to prove that the shape of the considered set is really close to round let us take $\lambda_1,\lambda_2$ with $|\lambda_1-\lambda_0|=|\lambda_2-\lambda_0|=r$. Then again by (\ref{pompom}) and Lemma~\ref{mdl} we obtain the following estimates
\bdm\aligned
\left|\frac{\xi_n(\lambda_1)-\mu_n(\lambda_0)}{\xi_n(\lambda_2)-\mu_n(\lambda_0)}\right|\leq& \frac{1+\varepsilon}{1-\varepsilon}\left|\frac{\xi_n(\lambda_1)-\mu_n(\lambda_1)}{\xi_n(\lambda_2)-\mu_n(\lambda_2)}\right|\leq \frac{(1+\varepsilon)^2}{(1-\varepsilon)^2}\left|\frac{(f^n_{\lambda_1})'(e_{\lambda_1})x(\lambda_1)}{(f^n_{\lambda_2})'(e_{\lambda_2})x(\lambda_2)}\right|\leq\\ 
\leq&\frac{(1+\varepsilon)^3}{(1-\varepsilon)^2}\left|\frac{(f^n_{\lambda_2})'(e_{\lambda_2})x(\lambda_1)}{(f^n_{\lambda_2})'(e_{\lambda_2})x(\lambda_2)}\right|= \frac{(1+\varepsilon)^3}{(1-\varepsilon)^2}\left|\frac{x(\lambda_1)}{x(\lambda_2)}\right|.
\endaligned \edm
The last expression can be arbitrarily close to~$1$ independently of $n$ for small $r$. This means that the set $\xi_n(\gamma_r)$ is close to a~circle centered at $\xi_n(\lambda_0)=\mu_n(\lambda_0)$ and of radius $|\xi_n(\lambda)-\mu_n(\lambda_0)|$ for any $|\lambda-\lambda_0|=r$, so the annulus $A$ is mapped onto a~slightly distorted annulus whose shape can be controlled independently of~$n$. This finishes the proof of the lemma. 
\end{proof}

With the notation of the previous lemma, we obtain from its proof and Lemma~\ref{mdl} the following important corollary.
\begin{corollary}\label{ballin} If $n$ is maximal for which $|\xi_n(\lambda)-\mu_n(\lambda)|\leq\delta'$, $\lambda\in{B(\lambda_0,r_2)}$, then for all $\lambda$ with $|\lambda-\lambda_0|=r_2$ we have $|\xi_n(\lambda)-\mu_n(\lambda)|\geq\frac{\delta'}{2M'}$, if $\delta'>0$ and $r>0$ were chosen small enough. \end{corollary}

\subsection{Measure estimates}\label{secmeasest}

By now we know how to control the behaviour of $\xi_n$ in a~small scale. In this section we will derive measure estimates in a~large scale, i.e. when a~parametric ball attains under~$\xi_n$ some fixed size. 
Recall that we consider $f_{\lambda}$, $\lambda\in{B(\lambda_0,\varepsilon)}$, for some small $\varepsilon>0$ and $\lambda_0$ is the~parameter satisfying assumptions of Theorem~\ref{A}. Assuming that $r\leq\varepsilon$ is so small that $z$ and its holomorphic motion $h_{\lambda}(z)$ are close enough for all $z\in\HS$ and $\lambda\in{B(\lambda_0,r)}$, we get from Lemma~\ref{annulus} and Corollary~\ref{ballin} the~following fact.
\begin{proposition}\label{elipdistprop} 
There exist $\delta'>0$ and $0<r<\varepsilon$, depending only on~$f_{\lambda_0}$, such that for any $0<r_2<r$, if $n$~is the~biggest number for which $\diam(\xi_n(B(\lambda_0,r_2)))\leq\delta'$, then we can find two discs $D_1$ i $D_2$ such that $D_1\subset{D}\subset{D_2}$, where $D=\xi_n(B(\lambda_0,r_2))$, with the~following properties:
$$ \frac{\diam(D_2)}{\diam(D_1)}=4M' \,,\; \diam(D_1)=\frac{\delta'}{M'} $$
and $D_1$ is centered at~$\mu_n(\lambda_0)\in{J(f_{\lambda_0})}$. The~degree of~$\xi_n$ on~$B(\lambda_0,r)$ is bounded above by~$K$, depending only on the~family~$f_{\lambda}$, $\lambda\in{B(\lambda_0,\varepsilon)}$.
\end{proposition}

The~next step is to estimate the~Lebesgue measure of those parameters~$\lambda$ for which some iterate $f_{\lambda}^n(e_{\lambda})$ either turns back to a~neighbourhood of a~critical point or escapes close to infinity. First, however, we need to know how many iterates are required to cover a~neighbourhood of infinity and critical points
\begin{equation}\label{udelta}  
	U_{\delta}={B\big(\Crit(f_{\lambda_0}),\delta\big)}\cup{B\left(\infty,\delta\right)},
\end{equation}
for an~arbitrary small $\delta>0$. To be precise, we want to estimate the~number of iterates of~$f_{\lambda}$, $\lambda\in{B(\lambda_0,r)}$ for some $r>0$, after which the image of a~small disk intersecting the Julia set covers $U_{\delta}$.

Recall that the Julia set $J(f_{\lambda})$ is the~closure of prepoles of~$f_{\lambda}$ (see e.g.~\cite{berg}), thus any open disc intersecting the Julia set after finite number of steps will cover under~$f_{\lambda}$ the whole~$\Cd$ (elliptic functions have no omitted values). Moreover, since poles move holomorphically with the~parameter~$\lambda$, the~number of steps is locally constant in the~parameter plane.

\begin{lemma}\label{freeperiod} 
Let $D$ be an~open and bounded set disjoint from $U_{\delta}$ containing an~open disk of radius $d>0$ centered at the~Julia set of some $f=f_{\lambda}$. Then we can choose an~$N$, depending only on $d$, $f$ and $U_[\delta]$, such that
\bdm \inf\left\{m\in\N: f^m(D)\supset\overline{U_{\delta}}\right\}\leq{N}. \edm 
\end{lemma}
\begin{proof}
Cover $\overline{J(f)\setminus{U_{\delta}}}$ with a~collection of open disks $D_z$ of diameter $d$ centered at $z\in\overline{J(f)\setminus{U_{\delta}}}$. Since the prepoles of~$f$ are dense in~$J(f)$, for every $D_z$ there is a~minimal $n=n(z)$ such that 
\bdm f^n(D_z)\supset\overline{U_{\delta}}. \edm
But $n(z)$ is constant in some neighbourhood of $z$ since $f^n$ is continuous, moreover $\overline{J(f)\setminus{U_{\delta}}}$ is compact in~$\C$, therefore we can find an~integer~$N$ such that $n(z)\leq{N}$ for every~$z$.
\end{proof}

Note that we can choose a~radius $r>0$ so that the~statement holds for every $f_{\lambda}$, $\lambda\in{B(\lambda_0,r)}$ and possibly slightly bigger $N$, which depends only on $d>0$ for $r$ small enough. It is possible since the dependence on $\lambda$ is analytic hence continuous.

We know now that $f^m(D)\Supset{U_{\delta}}$ for some $m\leq{N}$. We will estimate the~measure of those points from $D$ that get mapped into $U_{\delta}$ under~$f^j$ for some $j\leq{m}$. Recall that $f=f_{\lambda}$ is a~Weierstrass elliptic function and $D$ is an~open and bounded set disjoint from $U_{\delta}$. In particular $D\cap{B(\infty,\delta)}=\emptyset$. The following lemma is similar to an~analogous one in the~rational case (cf. {\cite[Lemma~4.2]{a}}) and for the~exponential family \cite{b}, however because of the presence of poles we need to be much more careful. Let $\mu$ denotes the~Lebesgue measure on the~Riemann sphere~$\Cd$ and recall that the~derivatives are spherical and $U_{\delta}$ is given by (\ref{udelta}).

\begin{lemma}\label{measest} 
Assume that $D$ is an~open set disjoint from $U_{\delta}$ and $f^m(D)\Supset{U}_{\delta}$ for some integer $m$. Then there exists a~constant $C>0$, depending only on $f$, $m$ and~$U_{\delta}$, such that
$$ \mu\left(\left\{z\in{D}\colon f^j(z)\in{U_{\delta}}\textrm{ for some } 1\leq{j}\leq{m}\right\}\right)\geq C\mu(D). $$
\end{lemma}
\begin{proof}
Let us define
\bdm F=\{z\in{D}: f^j(z)\in{U_{\delta}}\textrm{ for some } 1\leq{j}\leq{m}\} \edm
Divide $F$ into $m$ pairwise disjoint subsets, i.e. domains of the first entry map to $U_{\delta}$:
\bdm\aligned
F_1&=\{z\in{D}:f(z)\in{U_{\delta}}\}=f^{-1}(U_{\delta})\cap{D},\\
F_2&=\{z\in{D}:f^2(z)\in{U_{\delta}}\textrm{ but }f(z)\notin{U_{\delta}}\}=f^{-2}(U_{\delta})\cap{f^{-1}\left(\Cd\setminus{U_{\delta}}\right)}\cap{D},\\
F_3&=\{z\in{D}:f^3(z)\in{U_{\delta}}\textrm{ but }f(z)\notin{U_{\delta}}, f^2(z)\notin{U_{\delta}}\},\\
&\vdots\\
F_m&=\{z\in{D}:f^m(z)\in{U_{\delta}}\textrm{ but }f^j(z)\notin{U_{\delta}}\textrm{ for }j\leq{m-1}\}=\\
&=f^{-m}(U_{\delta})\cap\bigcap_{j=1}^{m-1}f^{-j}\left(\Cd\setminus{U_{\delta}}\right)\cap{D}.
\endaligned\edm
Then $F=F_1\cup{F_2}\cup\ldots\cup{F_m}$ and the union is disjoint. Moreover, since $D$ is bounded, the~definition assures that for any $j=1,\ldots,m$, the~set $F_j$ contains no essential singularities of $f^j$ so the~spherical derivative of~$f^j$ is well defined everywhere in~$F_j$.
Notice also that
$$ D\setminus{F}=\{z\in{D}:f(z)\notin{U_{\delta}},\ldots,f^m(z)\notin{U_{\delta}}\}= \bigcap_{j=1}^{m}f^{-j}\left(\Cd\setminus{U_{\delta}}\right)\cap{D}. $$
Since $\Cd\setminus{U_{\delta}}$ is bounded, the~set $D\setminus{F}$ contains no poles of any $f^j$ for $j=1,\ldots,m$, hence also no essential singularity of~$f^m$.

To estimate the~degree of $f^m$ on $D\setminus{F}$ recall that $f$ is periodic with respect to an~appropriate lattice and on every period
parallelogram the degree of~$f$ equals~two. The~set $\Cd\setminus{U_{\delta}}$ is bounded in~$\C$, i.e. it is contained in $\Cd\setminus{B(\infty,\delta)}$, so it intersects finitely many, say $n_{\delta}$, period parallelograms. Hence the degree of~$f$ on $\Cd\setminus{U_{\delta}}$ is bounded by~$2n_{\delta}$. Now, every iterate of~$f$ that we consider maps a~subset of~$\Cd\setminus{U_{\delta}}$ back into $\Cd\setminus{U_{\delta}}$, thus the~degree of~$f^2$ is bounded by $(2n_{\delta})^2$  on the~set $f^{-1}\left(\Cd\setminus{U_{\delta}}\right)\cap(\Cd\setminus{U_{\delta}})$, etc. We conclude that the degree of $f^m$ on $D\setminus{F}$ is at most $(2n_{\delta})^m$ and this number depends only on~$f$, $m$ and $\delta$.

Moreover, on every $F_j$ the spherical derivative $|(f^j)'|$ is bounded from above by some constant $c_j=c_j(f,m,\delta)$. On the other hand on $D\setminus{F}$, $|(f^m)'|$ is bounded from below by a~constant $a=a(f,m,\delta)>0$ (there are neither poles nor essential singularities of $f^m$ and we are far away from $\Crit(f^m)$). We get the~following estimates.
\beq\label{1} \mu(U_{\delta})\leq\sum_{j=1}^{m}\int\limits_{F_j}|(f^j)'(z)|^2d\mu(z)\leq \sum_{j=1}^{m}c_j^2\,\mu(F_j)\leq \max_{j=1,\ldots,m}c_j^2\;\sum_{j=1}^{m}\mu(F_j)=: C_1\mu(F), \eeq
Denote $g(w)=\{z\in{D\setminus{F}}:f^m(z)=w\}$ for $w\in\Cd\setminus{U_{\delta}}$. Then:
\begin{equation}\label{2}\aligned
 \mu(D\setminus{F})=\int\limits_{\Cd\setminus{U_{\delta}}}\sum_{z\in{g(w)}}|(f^m)'(z)|^{-2}d\mu(w)\leq (2n_{\delta})^ma^{-2}\mu\left(\Cd\setminus{U_{\delta}}\right) =:\kappa\;\mu\left(\Cd\setminus{U_{\delta}}\right). 
\endaligned\end{equation}
Finally, for some constant $M_{\delta}$, depending only on~$\delta$, we have that
\begin{equation}\label{3} 
	\mu(U_{\delta})\geq M_{\delta}\;\mu\left(\Cd\setminus{U_{\delta}}\right). 
\end{equation}
Putting together (\ref{1}), (\ref{2}) and (\ref{3}) we obtain the following
$$ \mu(F)\geq \frac{1}{C_1}\;\mu(U_{\delta})\geq \frac{M_{\delta}}{C_1}\mu(\Cd\setminus{U_{\delta}})\geq \frac{M_{\delta}}{C_1\kappa}\;\mu(D\setminus{F}), $$
which implies that 
$$ \mu(F)\geq C\mu(D) $$ 
for some constant $C=C(f,m,\delta)$.
\end{proof}

\subsection{Conclusion}
 
To conclude with the~proof of Theorem~\ref{B}, recall that $f_{\lambda_0}$ was a~Weierstrass elliptic function from $\W_t$ or $\W_s$ with $\lambda_0\in\M$ and consider nearby maps $f_{\lambda}$, $\lambda\in{B(\lambda_0,r)}$ for some small $r>0$. Take an~arbitrarily small $\delta>0$ (such that e.g. $\lambda_0\in\M_{\delta}$). We want to show that the set $\M_{\delta}$ has the Lebesgue density less than one at~$\lambda_0$. 

We will assume that $r>0$ is so small that critical points of $f_{\lambda}$, $\lambda\in{B(\lambda_0,r)}$, are $\delta/4$ close to appropriate critical points of~$f_{\lambda_0}$ -- it is possible since critical points depend analytically on~$\lambda$ and we have only finitely many periodic families of critical points for Weierstrass elliptic functions. Then we have that
\begin{equation}\label{udeltainclusion}
	\forall\;\lambda\in{B(\lambda_0,r)}\qquad U_{3\delta/4}\subset B\big(\Crit(f_{\lambda}),\delta\big)\cup B(\infty,\delta),
\end{equation}
where $U_{\delta}$ is given by~$(\ref{udelta})$. In what follows we will estimate the~Lebesgue measure of the~set of parameters~$\lambda$ for which some iterate of a~critical value~$e_{\lambda}$ falls into~$U_{3\delta/4}$, hence $\lambda\notin\M_{\delta}$.

Let $\delta'>0$ and $r>0$ be chosen so that the statement of Proposition~\ref{elipdistprop} is satisfied and all our expansion and distortion properties hold. Consider a~parameter ball $B=B(\lambda_0,r_2)$ for any $r_2\leq{r}$ and let $n$ be the~largest integer for which the~set $D:=\xi_n(B)$ has the~diameter at most $\delta'$. Let the~discs $D_1\subset{D}\subset{D_2}$ are as in Proposition~\ref{elipdistprop}.

Lemma~\ref{freeperiod} implies that there exists an~$N>0$ such that $f_{\lambda_0}^m(D_1)\Supset{U_{\delta/2}}$ for some $m\leq{N}$ independently of the~center of~$D_1$. Because of the~inclusion $D_1\subset{D}\subset{D_2}$ and since $\diam(D_2)/\diam(D_1)=4M'$ we get by Lemma~\ref{measest} that
\begin{equation}\label{measest1} 
	\mu\left(\left\{z\in{D}:f_{\lambda_0}^m(z)\in{U}_{\delta/2}\right\}\right) \geq C_1\mu(D) 
\end{equation}
for some constant $C_1$ depending only on the~family $f_{\lambda}$, the~set $U_{\delta}$ and~$N$. 
Since we have only finitely many steps to consider we can decrease, if necessary, the~radius $r>0$ so that for every $\lambda\in{B(\lambda_0,r)}$,
$$	f_{\lambda_0}^m(\xi_n(\lambda))\in{U}_{\delta/2} \;\Longrightarrow\; \xi_{n+m}(\lambda)=f_{\lambda}^m(\xi_n(\lambda))\in{U}_{3\delta/4} $$
for any $m\leq{N}$.

\begin{lemma}\label{claim} 
It is possible to choose $\delta''\in(0,\delta')$ so that for every radius $0<r_2<r$ and all $\lambda\in{B(\lambda_0,r_2)}$,
$$ \xi_{n+j}(\lambda)\in{U_{3\delta/4}} \;\textrm{ for some }\, j\leq{N} \;\Longrightarrow\; \lambda\in{A(\lambda_0;r_1,r_2)}, $$
where $r_1>0$ is minimal for which $|\xi_n(\lambda)-\mu_n(\lambda)|\geq\delta''$ for all $\lambda\in{A(\lambda_0;r_1,r_2)}$. 
\end{lemma}
\begin{proof}
We can choose $\delta''>0$ as small as desired provided $r>0$ is small enough. Thus, to have that for any $\lambda\in{B(\lambda_0,r)}$ with $|\xi_n(\lambda)-\mu_n(\lambda)|\leq\delta''$ and for all $j\leq{N}$,
$$  |\xi_{n+j}(\lambda)-\mu_{n+j}(\lambda)|\leq b^j|\xi_n(\lambda)-\mu_n(\lambda)|\leq \delta'  $$
it is sufficient to choose $\delta''$ so small that $b^{N}\leq\frac{\delta'}{\delta''}$, where 
$$  b=\max\{|f_{\lambda}'(z)|:z\in\NN, \lambda\in{B(\lambda_0,r)}\}\,, \;1<b<\infty.  $$
Next, we know that $\mu_{n+j}(\lambda)\in\HS_{\lambda}\subset\NN$ (if $r$ is small) and $\NN\cap{U_{\delta}}=\emptyset$. Therefore, if $\delta'<\delta/4$, then $\xi_{n+j}(\lambda)\notin{U_{3\delta/4}}$ for all $\lambda$ satisfying $|\xi_n(\lambda)-\mu_n(\lambda)|\leq\delta''$.
\end{proof}
We get the~following inclusions
\begin{equation}\label{inclusions} 
	A(\lambda_0;r_1,r_2) \supset \left\{\lambda\in{B}:\xi_{n+m}(\lambda)\in{U_{3\delta/4}}\right\} \supset \xi_n^{-1}\left(\left\{z\in{D}:f_{\lambda_0}^m(z)\in{U}_{\delta/2}\right\}\right). 
\end{equation}

Recall that inside the~annulus $A=A(\lambda_0;r_1,r_2)$ we have bounded distortion of $\xi_n$:
$$ \frac{1}{C'}\left(\frac{r_1}{r_2}\right)^{K-1}\leq \left|\frac{\xi_n'(\lambda_1)}{\xi_n'(\lambda_2)}\right|\leq C'\left(\frac{r_2}{r_1}\right)^{K-1}. $$
Moreover, if $r>0$ was chosen small enough and we take any two parameters $\lambda_i$ with $|\lambda_i-\lambda_0|=r_i$, $i=1,2$, then since $\diam(\xi_n(B))\leq\delta'$,
$$ |\xi_n(\lambda_2)-\mu_n(\lambda_2)|\leq\frac{1}{1-\varepsilon}\delta', $$
and by the~choice of~$r_1$
\bdm |\xi_n(\lambda_1)-\mu_n(\lambda_1)|\geq\delta''.  \edm
Consequently, applying Lemma~\ref{mdl} and (\ref{pom}), we get similarly like in the proof of Lemma~\ref{annulus},
\bdm\aligned
 \frac{\delta''}{\delta'}\leq& \frac{1}{1-\varepsilon}\left|\frac{\xi_n(\lambda_1)-\mu_n(\lambda_1)}{\xi_n(\lambda_2)-\mu_n(\lambda_2)}\right|\leq \frac{1+\varepsilon}{(1-\varepsilon)^2}\left|\frac{(f_{\lambda_1}^n)'(e_{\lambda_1})x(\lambda_1)}{(f_{\lambda_2}^n)'(e_{\lambda_2})x(\lambda_2)}\right| \leq\\ \leq&\frac{(1+\varepsilon)^2}{(1-\varepsilon)^2}\left|\frac{(f_{\lambda_2}^n)'(e_{\lambda_2})x(\lambda_1)}{(f_{\lambda_2}^n)'(e_{\lambda_2})x(\lambda_2)}\right|= \frac{(1+\varepsilon)^2}{(1-\varepsilon)^2}\left|\frac{x(\lambda_1)}{x(\lambda_2)}\right|\leq \frac{(1+\varepsilon)^3}{(1-\varepsilon)^3}\left(\frac{r_1}{r_2}\right)^K,
\endaligned\edm
and therefore
$$ \left(\frac{r_1}{r_2}\right)^K\geq \left(\frac{1-\varepsilon}{1+\varepsilon}\right)^3 \frac{\delta''}{\delta'}. $$

As a consequence we obtain uniform bounds on the distortion of $\xi_n$ on the annulus $A$:
\beq\label{bdist} \tilde{C}^{-1}\leq \left|\frac{\xi_n'(\lambda_1)}{\xi_n'(\lambda_2)}\right|\leq \tilde{C} \eeq
for all $\lambda_1,\lambda_2\in{A}$, where $\tilde{C}$ depends only on $\delta''$ and $\delta'$.

\medskip
In order to estimate the~Lebesgue measure of the~set $\{\lambda\in{B(\lambda_0,r_2):\xi_{n+m}}\in{U}_{3\delta/4}\}$ for any radius $0<r_2\leq{r}$ and appropriate~$m\leq{N}$ let us denote
$$  E=\{z\in{D}:f_{\lambda_0}^m(z)\in{U}_{\delta/2}\}  $$
and fix an~arbitrary point $z_0\in{A}$. By (\ref{inclusions}) we have $\xi_n^{-1}(E)\subset{A}$ and hence by (\ref{bdist})
$$ \mu(E)\leq\int_{\xi_n^{-1}(E)}|\xi_n'(z)|^2d\mu(z)\leq \tilde{C}^{2}|\xi_n'(z_0)|^2\mu(\xi_n^{-1}(E)). $$

On the~other hand, since the degree of $\xi_N$ is bounded by $K$ on~$A$,
\bdm \mu(A)=\int\limits_{D}\sum_{z\in\xi_n^{-1}(w)\cap{A}}|\xi_n'(z)|^{-2}d\mu(w)\leq \tilde{C}^2K|\xi_n'(z_0)|^{-2}\mu(D). \edm
Therefore, by (\ref{measest1}) and since $r_1/r_2\leq0.1$ (see Lemma~\ref{annulus}), we get the~following inequalities
$$\aligned 
\mu(\xi_n^{-1}(E))\geq& \tilde{C}^{-2}|\xi_n'(z_0)|^{-2}\mu(E)\geq \tilde{C}^{-2}|\xi_n'(z_0)|^{-2}C\mu(D)\geq\\ 
\geq& \frac{C\tilde{C}^{-4}}{K}\mu(A)\geq \frac{C\tilde{C}^{-4}}{K}\frac{99}{100}\mu(B). 
\endaligned$$

Thus for some $q\in(0,1)$, $q=q(\delta',\delta'',\delta)$, we have that
\bdm \mu\left(\xi_n^{-1}(E)\right)\geq q\mu(B). \edm
By (\ref{inclusions}) this implies that
\bdm \mu\left(\{\lambda\in{B}:\xi_j(\lambda)\in{U_{3\delta/4}} \textrm{ for some } j\geq{n}\}\right)\geq q\mu(B). \edm
By (\ref{udeltainclusion}) if the~critical value $e_{\lambda}$ falls under $f_{\lambda}$ to $U_{3\delta/4}$, then the parameter $\lambda$ cannot be in $\M_{\delta}$, so
\bdm \mu\left(\left\{\lambda\in{B(\lambda_0,r_2)}: \lambda\notin\M_{\delta}\right\}\right)\geq q\mu(B(\lambda_0,r_2)).  \edm
Since it holds for an~arbitrary small $r_2\leq{r}$, the~Lebesgue density of the~set $\M_{\delta}$ at~$\lambda_0$ is at most $1-q<1$. This finishes the proof of Theorem~\ref{B}.

\section{Proof of the~Lemma~\ref{critpole}}

To finish the~proof of Theorem~\ref{main} we need to deal with the~case when all critical values are prepoles. Recall first that every Weierstrass elliptic function has a~countable family of poles, which are exactly lattice points. Poles of~$f_{\lambda}$ are given by
$$ p_{j,k}(\lambda)=j\lambda+k\e^{2\pi{i}/3}\lambda, \quad j,k\in\Z $$
for $f_{\lambda}\in\W_t$ and by
$$ p_{j,k}(\lambda)=j\lambda+ki\lambda, \quad j,k\in\Z $$
for $f\in\W_s$. These are obviously analytic functions of~$\lambda$.

Suppose now that $\lambda_0$ is a~parameter for which all critical values of $f_{\lambda_0}\in\W_t\cup\W_s$ are prepoles, i.e.
\begin{equation}\label{critpoleeq}
	f_{\lambda_0}^n(e_{\lambda_0})=p_{j,k}(\lambda_0)
\end{equation}
for some $n\geq0$. In case of a~triangle lattice $e_{\lambda_0}$ is any of the~three critical values (then for remaining critical values we have analogous equations multiplied by $\e^{2\pi{i}/3}$ and $\e^{4\pi{i}/3}$ respectively) while for a~square lattice we take $e_{\lambda_0}\neq0$.

Consider the~following function
$$ g(\lambda)=f_{\lambda}^n(e_{\lambda})-p_{j,k}(\lambda) $$
in a~neighbourhood of~$\lambda_0$, where numbers $j,k\in\Z$ and $n\in\N$ are fixed. It is a~holomorphic function of~$\lambda$ for $\lambda$ close to~$\lambda_0$ and by~(\ref{critpoleeq}) we have $g(\lambda_0)=0$. We have two cases: either $g$ is an~open map and $\lambda_0$ is its isolated root or $g(\lambda)\equiv0$ locally.

If the~second condition holds, for all parameters $\lambda$ close to~$\lambda_0$ the~dynamics of critical values is the~same. To be precise, all critical values of~$f_{\lambda}$ are mapped onto fixed poles after fixed number of iterates. We can argue exactly like in the~proof of transversality condition (Lemma~\ref{transver}) -- parameter $\lambda_0$ is postsingularly stable and we can find a~conjugacy between $f_{\lambda}$ and $f_{\lambda_0}$ defined on branches of consecutive preimages of critical values. The~conjugacy may be extended to a~quasiconformal map on the~Julia set $J(f_{\lambda_0})$ conjugating $f_{\lambda_0}$ with~$f_{\lambda}$ for all $\lambda$ close to~$\lambda_0$. There exists, therefore, on $J(f_{\lambda_0})$ an~$f_{\lambda_0}$-invariant line-field contrary to {\cite[Theorem~1.1]{rvs}} (cf. {\cite[Theorem~2]{gks}}). This case cannot happen.

It implies that $g$ is not constant and hence $\lambda_0$ is its isolated root. Consequently, theres is no $\lambda$ close to~$\lambda_0$ for which critical values of~$f_{\lambda}$ are eventually mapped onto these poles after $n$ iterates (in the~case of a~square lattice this does not concern $0$ which is always a~pole), hence the~set of parameters satisfying (\ref{critpoleeq}) is discrete. Since there are only countably many such equations, we conclude that the~set of parameters~$\lambda$ for which all critical values of~$f_{\lambda}$ are prepoles is countable. This finishes the~proof of~Lemma~\ref{critpole}.

Notice that this does not prove that the~whole set of parameters for which all critical values are prepoles is discrete. Moreover, results of Jane Hawkins and her collaborates show that these parameters accumulate similarly to a~family of consecutive prepoles of a~meromorphic function. Still, they form a~countable set whose Lebesgue measure in~$\C$ equals zero.

\bigskip
\textbf{Acknowledgments}
The author would like to thank Jane Hawkins for helpful and encouraging discussions concerning elliptic functions.

\end{document}